\documentclass[10pt]{article}

\usepackage[a4paper,margin=1in]{geometry}
\usepackage[T1]{fontenc}
\usepackage{lmodern}
\usepackage{amsmath,amssymb,amsthm,mathtools}
\usepackage{microtype}
\usepackage[skip=0.45\baselineskip plus 2pt,indent=0pt]{parskip}
\usepackage{aliascnt}
\usepackage{needspace}
\usepackage[colorlinks=true,linkcolor=blue,citecolor=blue,urlcolor=blue]{hyperref}
\usepackage[capitalize,nameinlink,noabbrev]{cleveref}

\numberwithin{equation}{section}
\newtheorem{theorem}{Theorem}[section]
\newaliascnt{claim}{theorem}
\newtheorem{claim}[claim]{Claim}
\aliascntresetthe{claim}
\crefname{claim}{Claim}{Claims}
\Crefname{claim}{Claim}{Claims}
\crefname{theorem}{Theorem}{Theorems}
\Crefname{theorem}{Theorem}{Theorems}
\crefname{section}{Section}{Sections}
\Crefname{section}{Section}{Sections}

\DeclareMathOperator{\ex}{ex}
\newcommand{\E}{\mathbb E}
\newcommand{\Prb}{\mathbb P}

\hypersetup{
  pdftitle={Improved Sidorenko bound for hypergraph Turán densities},
  pdfauthor={Shuyan Chen, Xinmin Hou, Xizhi Liu},
  pdfsubject={Hypergraph Turán densities and the fixed-size Turán problem}
}

\title{\Large\bfseries Improved Sidorenko bound for hypergraph Tur\'an densities}
\author{%
Shuyan~Chen\thanks{\scriptsize Department of Mathematics, University of Manchester, Manchester M13 9PL, United Kingdom. Email:~\texttt{shuyan.chen-2@student.manchester.ac.uk}.}
\and
Xinmin~Hou\thanks{\scriptsize School of Mathematical Sciences, University of Science and Technology of China, Hefei 230026, China. \\
Hefei National Laboratory,
University of Science and Technology of China, Hefei 230088, Anhui, China. Email:~\texttt{xmhou@ustc.edu.cn}.}
\and
Xizhi~Liu\thanks{\scriptsize School of Mathematical Sciences, University of Science and Technology of China, Hefei 230026, China. Email:~\texttt{liuxizhi@ustc.edu.cn}.}}
\date{\today}

\begin{document}
\maketitle

\begin{abstract}
For every fixed integer $r\ge3$, we prove that the Tur\'an density of every $r$-uniform hypergraph $F$ with $m\ge1$ edges satisfies $\pi(F)\le1-c_rm^{-(r-1)/r}$, where $c_r>0$ depends only on $r$. This improves Sidorenko's general bound, whose gap from $1$ is of order $m^{-1}$. Complete $r$-uniform hypergraphs show that the exponent $(r-1)/r$ is optimal. The proof combines a potential argument for partial embeddings with the Lov\'asz Local Lemma.
\end{abstract}

\section{Introduction}

An \emph{$r$-uniform hypergraph} $F$, or \emph{$r$-graph}, consists of a finite vertex set $V(F)$ and a set of $r$-element subsets of $V(F)$ called edges. All hypergraphs considered here are simple. We identify a hypergraph with its edge set, and hence write $|H|$ for the number of edges of $H$. Write $v(F)\coloneqq|V(F)|$.

The \emph{Tur\'an number} $\ex(n,F)$ is the maximum number of edges in an $n$-vertex $r$-graph containing no copy of $F$. Katona, Nemetz, and Simonovits~\cite{KNS64} showed that the normalized Tur\'an numbers have a limit, called the \emph{Tur\'an density}, $\pi(F)\coloneqq\lim_{n\to\infty}{\ex(n,F)}/{\binom nr}$.

We study the largest possible Tur\'an density when the number of edges of the forbidden hypergraph is fixed. For an ordinary graph $F$ with $m\ge1$ edges, the theorem of Erd\H{o}s and Stone~\cite{ErdosStone}, in the form given by Erd\H{o}s and Simonovits~\cite{ErdosSimonovits}, yields $\pi(F)=1-1/(\chi(F)-1)$, where $\chi(F)$ is the chromatic number of $F$. Since $m\ge\binom{\chi(F)}2$, we have $1-\pi(F)\ge(2m)^{-1/2}$. Complete graphs show that the exponent $1/2$ is optimal.

For $r\ge3$, a theorem of Erd\H{o}s~\cite{Erdos64} shows that an $r$-graph $F$ with at least one edge has $\pi(F)=0$ if and only if it is $r$-partite, meaning that $V(F)$ can be partitioned into $r$ parts so that every edge meets each part once. Positive Tur\'an densities are much harder to determine. Even for the complete $3$-graph $K_4^{(3)}$ on four vertices, the conjecture of Tur\'an~\cite{Turan41} that $\pi(K_4^{(3)})=5/9$ remains open. For further background on hypergraph Tur\'an problems, we refer the reader to the survey of Keevash~\cite{Keevash11}.

For upper bounds in terms of the number of forbidden edges, Sidorenko~\cite{Sidorenko89} proved that every $r$-graph $F$ with $m\ge2$ edges satisfies
\[
\pi(F)\le1-\frac1{m-1}.
\]
Sidorenko's later account~\cite{Sidorenko94} discusses this bound and the analytic methods behind it. Keevash~\cite[Theorem~1(ii)]{Keevash05} improved the bound by a lower-order term. For fixed $r\ge3$, his result gives, uniformly over $r$-graphs $F$ with $m$ edges as $m\to\infty$,
\[
\pi(F)<1-\frac1{m-1}
-\frac{1+o(1)}{2(r!)^{2/r}m^{3-2/r}}.
\]
Both bounds give a gap from $1$ of order $m^{-1}$. We improve this order to $m^{-(r-1)/r}$.

\begin{theorem}\label{thm:main}
For every integer $r\ge3$ and every $r$-graph $F$ with $m\ge1$ edges,
\[
\pi(F)\le1-
\frac1{4r^2}
\left(\frac4r\right)^{1/r}
\left(1-\frac1r\right)^{2-2/r}\frac{1}{m^{(r-1)/r}}.
\]
\end{theorem}

\emph{Remark.} A Lean~4 formalization establishes \Cref{thm:main}, together with the finite embedding assertion used in its proof and the existence of the Tur\'an density. The verified formalization is registered as \href{https://palomar-registry.org/entry?id=PALOMAR-2026-10-13-000012&version=1}{Palomar entry \texttt{PALOMAR-2026-10-13-000012}, version~1}, with source code at \url{https://github.com/xliu2022/improved-sidorenko-lean/tree/ba6f109bf01436ca15e55fdf5d35dbea85ffd9c9}.

For integers $s\ge r$, let $K_s^{(r)}$ be the complete $r$-graph on $s$ vertices. De Caen~\cite{deCaen83} proved that $\pi(K_s^{(r)})\le1-\binom{s-1}{r-1}^{-1}$. Sidorenko~\cite{Sidorenko81} proved the lower bound $\pi(K_s^{(r)})\ge1-((r-1)/(s-1))^{r-1}$ for $s\ge r$; see also his survey~\cite[Section~4.2]{SidorenkoSurvey}. For fixed $r$, $K_s^{(r)}$ has $m=\binom sr=\Theta_r(s^r)$ edges. Since $\binom{s-1}{r-1}^{-1}=s/(rm)$, the two bounds give $1-\pi(K_s^{(r)})=\Theta_r(m^{-(r-1)/r})$ as $s\to\infty$.

For integers $r\ge3$ and $m\ge1$, define
\[
\Pi_r(m)\coloneqq
\max\{\pi(F)\colon F\text{ is an $r$-graph with }|F|=m\}.
\]
It is clear that  deleting isolated vertices does not change the Tur\'an density. After their deletion, an $m$-edge $r$-graph has at most $rm$ vertices, so only finitely many isomorphism types occur and the maximum exists.

For fixed $r\ge3$ and any sufficiently large integer $m$, choose the largest integer $s\ge r$ with $\binom sr\le m$. Since $\binom sr\le m<\binom{s+1}r$, we have $m=\Theta_r(s^r)$ as $m\to\infty$. Add $m-\binom sr$ pairwise disjoint edges on new vertices to $K_s^{(r)}$, obtaining an $m$-edge $r$-graph $G$. Every $K_s^{(r)}$-free host is $G$-free, so Sidorenko's bound gives
\[
1-\pi(G)\le1-\pi(K_s^{(r)})
\le\left(\frac{r-1}{s-1}\right)^{r-1}
=O_r\bigl(m^{-(r-1)/r}\bigr).
\]

Combining \Cref{thm:main} with the bound on $1-\pi(G)$ above gives, for each fixed $r\ge3$ as $m\to\infty$,
\begin{equation}\label{eq:worst-case}
1-\Pi_r(m)=\Theta_r\bigl(m^{-(r-1)/r}\bigr).
\end{equation}

Keevash~\cite[Remark~(1)]{Keevash05} suggested that complete $r$-graphs should have the largest Tur\'an density among $r$-graphs with the same number of edges and observed that this assertion, together with de Caen's bound, would give a gap of order $m^{-(r-1)/r}$. \Cref{thm:main} establishes a lower bound of this order on $1-\pi(F)$ for every $r$-graph $F$.

For $m=3$, Frankl~\cite{Frankl90} proved that Sidorenko's bound $1/2$ is attained by the \emph{expanded triangle} of every even uniformity $r=2k\ge4$. An expanded triangle consists of the three pairwise unions of three disjoint $k$-sets. Keevash and Sudakov~\cite[Theorem~1.1]{KeevashSudakov05} showed that, for sufficiently large $n$, the extremal expanded-triangle-free $r$-graph consists of all $r$-sets meeting one part of a bipartition in an odd number of vertices, with the part sizes chosen to maximize the number of edges.
Hou, Liu, Zhang, Zhao, and Zhu~\cite[Theorem~1.1 and Corollary~1.2]{HouLiuZhangZhaoZhu25} proved that every three-edge $r$-graph $F$ with $r\ge3$ has $\pi(F)<1/2$ unless deleting its isolated vertices gives an expanded triangle,  answering negatively a question of Keevash~\cite[Remark~(2)]{Keevash05} which asks whether Sidorenko's bound $1/2$ can be attained for odd $r$. Other hypergraph Tur\'an problems with three-edge forbidden configurations include those for cancellative hypergraphs~\cite{Bollobas74,KeevashMubayi04}, generalized triangles~\cite{FranklFuredi89,Pikhurko08}, and weakly triangle-free hypergraphs~\cite{MubayiPikhurkoSudakov11,LiuMantel25}.

Other upper bounds depend on the intersections among the forbidden edges. Chen, Liu, Yang, and Zhu~\cite[Theorems~1.1 and~1.3]{ChenLiuYangZhu26} obtained upper bounds for the Tur\'an densities of certain hypergraphs with prescribed edge intersections in terms of the maximum size of set families with restricted intersection sizes or a bounded matching number. Using the Lov\'asz Local Lemma, Lu and Sz\'ekely~\cite[Theorem~4]{LuSzekely07} proved that $\pi(F)\le1-1/(\mathrm e(d+1))$ if every edge of $F$ meets at most $d$ other edges.

\section{Proof of the main theorem}\label{sec:proof}

We prove \Cref{thm:main} by combining a potential argument with the Local Lemma. We construct the embedding in two stages. In the first stage, we assign vertices whose incident unfinished edges have large total conditional failure probability, keeping all chosen vertices of the host hypergraph distinct and in a clique. One potential records the conditional failure probabilities of unfinished edges of $F$, and another bounds the proportion of host vertices excluded by the clique requirement. \Cref{clm:potential-updates} gives the estimates used in \Cref{lem:one-step} to assign one vertex. Iteration yields the partial map in \Cref{clm:small-load-map}, with small total conditional failure probability at every remaining vertex. In the second stage, we map the remaining vertices independently and apply the Local Lemma in \Cref{lem:local-completion} to ensure that every edge of $F$ is preserved.

\begin{proof}[Proof of \Cref{thm:main}]
Fix $r\ge3$ and set
\[
c_r\coloneqq
\frac1{4r^2}
\left(\frac4r\right)^{1/r}
\left(1-\frac1r\right)^{2-2/r}
=
\left(\frac{(r-1)^{2r-2}}{4^{r-1}r^{4r-1}}\right)^{1/r}.
\]
Let $F$ be an $r$-graph with $m\ge1$ edges, and fix $0<q<c_rm^{-(r-1)/r}$. We will show that, for all sufficiently large $n$, every $n$-vertex $r$-graph $H$ with $1-r!|H|/n^r\le q$ contains a copy of $F$.

The choice of $q$ implies $m^{r-1}q^r<c_r^r$. Substituting the value of $c_r$ and rearranging gives
\[
\left(\frac{4r^4}{(r-1)^2}mq\right)^{r-1}
<
\left(\frac{4r^4}{(r-1)^2}\right)^{r-1}\frac{c_r^r}{q}
=\frac1{r^3q}.
\]
We can therefore choose $T>0$ satisfying
\begin{equation}\label{eq:T-window}
\frac{4r^4}{(r-1)^2}mq
<T<
\left(\frac1{r^3q}\right)^{1/(r-1)}.
\end{equation}
Put $M\coloneqq r^2mq/(r-1)$. By the lower bound on $T$, we have $1/r-4M/T>1/r-(r-1)/r^2=1/r^2$. We can therefore fix $\beta$ satisfying
\begin{equation}\label{eq:beta-choice}
\frac1{r^2}<\beta<\frac1r-\frac{4M}{T}.
\end{equation}

Let $H$ be an $n$-vertex $r$-graph with $q_H\coloneqq1-r!|H|/n^r\le q$. Since $\beta>1/r^2$, we can take $n$ sufficiently large that
\begin{equation}\label{eq:n-choice}
\frac{v(F)}n\le\beta-\frac1{r^2}
\quad\text{and}\quad
\frac{\binom{v(F)}2}{n}
<
\left(1-\frac1{2r}\right)^m.
\end{equation}
We construct the embedding by mapping the vertices of $F$ to vertices of $H$. Once all vertices of an edge of $F$ have been mapped, we call the edge \emph{successful} if they are mapped to distinct vertices forming an edge of $H$, and say that it \emph{fails} otherwise. To record failure for an ordered $r$-tuple of vertices of $H$, define $\gamma_r\colon V(H)^r\to\{0,1\}$ by
\[
\gamma_r(x_1,\ldots,x_r)\coloneqq
\begin{cases}
0,&\text{if the $x_i$ are distinct and }\{x_1,\ldots,x_r\}\in H,\\
1,&\text{otherwise}.
\end{cases}
\]
Let $X_1,\ldots,X_r$ be independent uniform vertices of $H$. Of the $n^r$ equally likely ordered tuples, exactly $r!|H|$ are successful, so $\E[\gamma_r(X_1,\ldots,X_r)]=1-r!|H|/n^r=q_H$. For $0\le j<r$ and $a_1,\ldots,a_j\in V(H)$, define the conditional failure probability
\[
\gamma_j(a_1,\ldots,a_j)\coloneqq
\E[\gamma_r(a_1,\ldots,a_j,X_{j+1},\ldots,X_r)].
\]
Thus $\gamma_0=q_H$. Each $\gamma_j$ is symmetric in its arguments. We write $\E_x$ for expectation over a uniform vertex $x$ of $H$. Averaging over the next coordinate gives, for $0\le j<r$,
\begin{equation}\label{eq:tower}
\E_x[\gamma_{j+1}(a_1,\ldots,a_j,x)]
=\gamma_j(a_1,\ldots,a_j).
\end{equation}

To prepare for the Local Lemma, we will first construct a partial map from $V(F)$ to $V(H)$ such that the conditional failure probabilities of the edges incident to each unassigned vertex have small sum. We build the partial map one vertex at a time, requiring the chosen vertices of $H$ to be distinct and to span a complete $r$-graph in $H$. Call a partial map satisfying this condition \emph{legal}. Legality ensures that every completely assigned edge of $F$ is successful.

Consider a legal partial map $\varphi$ from $V(F)$ to $V(H)$, and let $U\subseteq V(F)$ be the set of unassigned vertices. For each $e\in F$, put $s_e\coloneqq|e\cap U|$, and call $e$ \emph{unfinished} if $s_e\ge1$. For each unfinished edge $e$, let $p_e$ be the probability that $e$ fails when its unassigned vertices are mapped independently and uniformly to $V(H)$. If $\varphi$ maps the assigned vertices of $e$ to $a_1,\ldots,a_{r-s_e}$, then $p_e=\gamma_{r-s_e}(a_1,\ldots,a_{r-s_e})$. We measure the total remaining failure probability, counting once for each unassigned vertex of an edge, by
\[
\Phi(\varphi)\coloneqq\sum_{\substack{e\in F\\s_e\ge1}}s_ep_e.
\]
We write $\Phi$ for $\Phi(\varphi)$ when the partial map is fixed. For $v\in U$, let $\ell_v\coloneqq\sum_{e\in F\colon v\in e}p_e$ be the sum of the conditional failure probabilities of the edges of $F$ containing $v$. Then $\Phi=\sum_{v\in U}\ell_v$, and the initial value is $\Phi_0=rmq_H$.

Let $\mathbf a=(a_1,\ldots,a_k)$ be the list of vertices of $H$ selected by $\varphi$. Since $\varphi$ is legal, these vertices are distinct and span a complete $r$-graph. A vertex of $H$ outside $\{a_1,\ldots,a_k\}$ is a legal choice for the next assignment precisely when it forms an edge of $H$ with every $(r-1)$-subset of this set. We use the following sums to bound the proportion of vertices that fail this edge condition. Write $[k]\coloneqq\{1,\ldots,k\}$, with $[0]=\emptyset$, and use $\binom{[k]}j$ for the family of its $j$-element subsets. Define
\[
\begin{aligned}
S_j(\mathbf a)\coloneqq
\sum_{I\in\binom{[k]}j}\gamma_j((a_i)_{i\in I})
\quad\text{for }0\le j\le r-1,\qquad\text{and}\qquad \Xi(\mathbf a)\coloneqq
\sum_{j=0}^{r-1}T^{r-1-j}S_j(\mathbf a).
\end{aligned}
\]
Since each $\gamma_j$ is symmetric, reordering its arguments does not change its value. We omit the argument $\mathbf a$ from $S_j(\mathbf a)$ and $\Xi(\mathbf a)$ when the partial map is fixed. For the empty partial map, $S_0=q_H$ and $S_j=0$ for $1\le j\le r-1$, so the initial value of $\Xi$ is $\Xi_0=q_HT^{r-1}$. By the union bound, $S_{r-1}$ bounds the proportion of vertices of $H$ that fail to form an edge with some $(r-1)$-subset of $\{a_1,\ldots,a_k\}$. For $1\le j\le r-1$, appending a uniform random vertex of $H$ to $\mathbf a$ gives an expected increase of $S_{j-1}$ in $S_j$. Summing these expected increases with weights $T^{r-1-j}$ gives an expected increase in $\Xi$ of at most $\Xi/T$, as shown in \Cref{clm:potential-updates}(\ref{item:potential-xi}).

We combine the two potentials by defining
\[
\Psi\coloneqq\Phi+r^2M\Xi
=
\sum_{\substack{e\in F\\s_e\ge1}}s_ep_e
+r^2M\sum_{j=0}^{r-1}T^{r-1-j}
\sum_{I\in\binom{[k]}j}\gamma_j((a_i)_{i\in I}).
\]
Write $\Psi_{v\mapsto x}$ for its value after a legal assignment of $v\in U$ to $x\in V(H)$, and $\Psi_0$ for its initial value.

We will maintain $\Psi\le M$. Since $\Phi,\Xi\ge0$, this condition implies $\Xi\le1/r^2$. By the union bound and \eqref{eq:n-choice}, the proportion of host vertices that are not legal choices is at most $k/n+S_{r-1}\le v(F)/n+\Xi\le\beta$, so at least a proportion $1-\beta$ of the host vertices are legal choices. To verify the initial condition, we use $q_H\le q$ and the upper bound on $T$ in \eqref{eq:T-window}, which give
\[
\Phi_0\le rmq=\frac{r-1}{r}M,
\quad\text{and}\quad
r^2M\Xi_0\le r^2MqT^{r-1}<\frac Mr.
\]
Adding these two inequalities yields $\Psi_0<M$.

For an assignment $v\mapsto x$ with $v\in U$ and $x\in V(H)$, put $\mu_v\coloneqq\sum_{e\in F\colon v\in e}(s_e-1)p_e$. Each edge $e\in F$ containing $v$ has $s_e-1$ unassigned vertices after this assignment. The newly completed edges contribute $0$ to $\Phi$, so the new contribution of the edges containing $v$ is
\[
\tau_v(x)\coloneqq
\sum_{\substack{e\in F\\v\in e,\ s_e\ge2}}
(s_e-1)\gamma_{r-s_e+1}\bigl((\varphi(u))_{u\in e\setminus U},x\bigr).
\]
Here the value of $\gamma_{r-s_e+1}$ is the conditional failure probability of $e$ after assigning $v\mapsto x$. The expression for $\tau_v(x)$ is defined for every $x\in V(H)$, including choices that are not legal. For a legal assignment $v\mapsto x$, the contributions of edges not containing $v$ are unchanged. We therefore write the resulting value of $\Phi$ as
\[
\Phi_{v\mapsto x}\coloneqq
\sum_{\substack{e\in F\\v\notin e,\ s_e\ge1}}s_ep_e+\tau_v(x)
= \sum_{\substack{e\in F\\v\notin e,\ s_e\ge1}}s_ep_e + \sum_{\substack{e\in F\\v\in e,\ s_e\ge2}}
(s_e-1)\gamma_{r-s_e+1}\bigl((\varphi(u))_{u\in e\setminus U},x\bigr).
\]
Write $\Delta_x\Xi\coloneqq\Xi(\mathbf a,x)-\Xi(\mathbf a)$ for the change in the clique potential. Let $A\subseteq V(H)$ be the set of legal choices for $x$. We first estimate the potential updates and the size of $A$.

\Needspace{8\baselineskip}
\begin{claim}\label{clm:potential-updates}
For every legal partial map and every $v\in U$, the following statements hold.
\begin{enumerate}
\renewcommand{\theenumi}{\roman{enumi}}
\renewcommand{\labelenumi}{\textup{(\theenumi)}}
\item\label{item:potential-phi}
We have $\E_x[\tau_v(x)]=\mu_v\le(r-1)\ell_v$.
\item\label{item:potential-xi}
We have $\Delta_x\Xi\ge0$ for every $x\in V(H)$ and $\E_x[\Delta_x\Xi]\le\Xi/T$.
\item\label{item:potential-phi-update}
Every $x\in A$ satisfies $\Phi_{v\mapsto x}=\Phi-\ell_v-\mu_v+\tau_v(x)$.
\item\label{item:legal-choices}
The set $A$ of legal choices satisfies $\frac{|A|}{n}\ge1-\frac{v(F)}n-\Xi$.
\end{enumerate}
\end{claim}

\begin{proof}
For~(\ref{item:potential-phi}), since $1\le s_e\le r$ for every $e\in F$ containing $v$, we have $\mu_v\le(r-1)\ell_v$. Using \eqref{eq:tower} in the definition of $\tau_v(x)$ and noting that edges with $s_e=1$ have weight $0$, we obtain
\[
\begin{aligned}
\E_x[\tau_v(x)]
=\sum_{\substack{e\in F\\v\in e,\ s_e\ge2}}
(s_e-1)\E_x\!\left[
\gamma_{r-s_e+1}\bigl((\varphi(u))_{u\in e\setminus U},x\bigr)
\right]
=\sum_{\substack{e\in F\\v\in e,\ s_e\ge2}}(s_e-1)p_e
=\mu_v.
\end{aligned}
\]

For~(\ref{item:potential-xi}), appending $x$ to $\mathbf a$ leaves the terms indexed by subsets of $[k]$ unchanged. The additional terms in $S_j(\mathbf a,x)$ are indexed by the $j$-subsets of $[k+1]$ containing $k+1$. Thus, for $1\le j\le r-1$ and every $x\in V(H)$,
\[
S_j(\mathbf a,x)-S_j(\mathbf a)
=\sum_{I\in\binom{[k]}{j-1}}\gamma_j((a_i)_{i\in I},x).
\]
Since $S_0(\mathbf a,x)=S_0(\mathbf a)$, $T>0$, and all conditional failure probabilities are nonnegative, the definition of $\Xi$ gives
\[
\Delta_x\Xi
=\sum_{j=1}^{r-1}T^{r-1-j}\bigl(S_j(\mathbf a,x)-S_j(\mathbf a)\bigr)
=\sum_{j=1}^{r-1}T^{r-1-j}
\sum_{I\in\binom{[k]}{j-1}}\gamma_j((a_i)_{i\in I},x)
\ge0.
\]
By \eqref{eq:tower}, averaging over $x$ gives
\[
\E_x[S_j(\mathbf a,x)-S_j(\mathbf a)]
=\sum_{I\in\binom{[k]}{j-1}}\gamma_{j-1}((a_i)_{i\in I})
=S_{j-1}(\mathbf a).
\]
By linearity of expectation and these identities,
\[
\begin{aligned}
\E_x[\Delta_x\Xi]
=\sum_{j=1}^{r-1}T^{r-1-j}S_{j-1}
=\frac{\Xi-S_{r-1}}{T}
\le\frac{\Xi}{T}.
\end{aligned}
\]

For~(\ref{item:potential-phi-update}), the contribution of the edges containing $v$ before the assignment is $\sum_{e\in F\colon v\in e}s_ep_e=\ell_v+\mu_v$. After the assignment it is $\tau_v(x)$. All other contributions are unchanged, giving the stated formula for $\Phi_{v\mapsto x}$.

For~(\ref{item:legal-choices}), let $x$ be a uniform random vertex of $H$. If $x\notin A$, then either $x\in\{a_1,\ldots,a_k\}$ or $\gamma_r((a_i)_{i\in I},x)=1$ for some $I\in\binom{[k]}{r-1}$. The first event has probability $k/n$. For each fixed $I\in\binom{[k]}{r-1}$, the event $\gamma_r((a_i)_{i\in I},x)=1$ has probability $\gamma_{r-1}((a_i)_{i\in I})$ by the definition of $\gamma_{r-1}$. The union bound and the definition of $S_{r-1}$ therefore give
\[
\begin{aligned}
1-\frac{|A|}{n}
=\Prb[x\notin A]
&\le\Prb[x\in\{a_1,\ldots,a_k\}]
+\sum_{I\in\binom{[k]}{r-1}}
\Prb[\gamma_r((a_i)_{i\in I},x)=1]\\
&=\frac{k}{n}
+\sum_{I\in\binom{[k]}{r-1}}\gamma_{r-1}((a_i)_{i\in I})
=\frac{k}{n}+S_{r-1}.
\end{aligned}
\]
Since $k\le v(F)$ and $\Xi$ is the sum of $S_{r-1}$ and nonnegative terms, we obtain
\[
\frac{|A|}{n}\ge1-\frac{k}{n}-S_{r-1}
\ge1-\frac{v(F)}n-\Xi.\qedhere
\]
\end{proof}

We now use these estimates to choose a legal assignment that decreases $\Psi$.

\begin{claim}\label{lem:one-step}
Suppose that the current partial map is legal and $\Psi\le M$. If $v\in U$ satisfies $\ell_v\ge1/(4r)$, then some legal choice $x$ satisfies $\Psi_{v\mapsto x}<\Psi$.
\end{claim}

\begin{proof}
Since $\Phi,\Xi\ge0$ and $\Psi\le M$, we have $\Xi\le1/r^2$. By \Cref{clm:potential-updates}(\ref{item:legal-choices}) and \eqref{eq:n-choice}, using $\Xi\le1/r^2$ and $\beta<1/r<1$, we obtain
\[
\frac{|A|}{n}
\ge1-\frac{v(F)}n-\Xi
\ge1-\left(\beta-\frac1{r^2}\right)-\frac1{r^2}
=1-\beta>0.
\]
By \Cref{clm:potential-updates}(\ref{item:potential-phi-update}) and the definition of $\Delta_x\Xi$, we have, for every $x\in A$,
\[
\Psi_{v\mapsto x}=\Psi-\ell_v-\mu_v+\tau_v(x)+r^2M\Delta_x\Xi.
\]
By definition, $\tau_v(x)$ is a sum of nonnegative terms for every $x\in V(H)$. Since $x$ is uniform in $V(H)$ and $|A|/n\ge1-\beta>0$, \Cref{clm:potential-updates}(\ref{item:potential-phi}) and~(\ref{item:potential-xi}) give
\[
\begin{aligned}
\E_x[\tau_v(x)\mid x\in A]
&=\frac1{|A|}\sum_{x\in A}\tau_v(x)
\le\frac{n}{|A|}\E_x[\tau_v(x)]
\le\frac{\E_x[\tau_v(x)]}{1-\beta}
=\frac{\mu_v}{1-\beta},\\
\E_x[\Delta_x\Xi\mid x\in A]
&=\frac1{|A|}\sum_{x\in A}\Delta_x\Xi
\le\frac{n}{|A|}\E_x[\Delta_x\Xi]
\le\frac{\E_x[\Delta_x\Xi]}{1-\beta}
\le\frac{\Xi}{T(1-\beta)}.
\end{aligned}
\]
Substituting these bounds into the preceding formula for $\Psi_{v\mapsto x}$ and writing the fractions with denominator $T(1-\beta)$ gives
\begin{align*}
\E_x[\Psi_{v\mapsto x}\mid x\in A]
\le\Psi-\ell_v-\mu_v
+\frac{T\mu_v+r^2M\Xi}{T(1-\beta)}
=\Psi-\ell_v+\frac{\beta T\mu_v+r^2M\Xi}{T(1-\beta)}.
\end{align*}
We now use $\mu_v\le(r-1)\ell_v$. The bounds $\Xi\le1/r^2$ and $\ell_v\ge1/(4r)$ also give $r^2M\Xi\le M\le4rM\ell_v$. Substituting these two estimates into the preceding bound and collecting the coefficient of $\ell_v$, we obtain
\begin{align*}
\E_x[\Psi_{v\mapsto x}\mid x\in A]
\le\Psi-\ell_v
+\frac{\beta(r-1)T+4rM}{T(1-\beta)}\ell_v
=\Psi-\frac{(1-r\beta)T-4rM}{T(1-\beta)}\ell_v.
\end{align*}
By \eqref{eq:beta-choice}, we have $(1-r\beta)T-4rM>0$, so the coefficient of $\ell_v$ subtracted from $\Psi$ is positive. Hence the conditional expectation is less than $\Psi$, and some $x\in A$ satisfies the claim.
\end{proof}

We now iterate these assignments to obtain a partial map to which the Local Lemma can be applied.

\begin{claim}\label{clm:small-load-map}
There is a legal partial map with $\Psi<M$ such that $\ell_v<1/(4r)$ for every unassigned vertex $v$.
\end{claim}

\begin{proof}
Starting with the empty partial map, for which $\Psi_0<M$, apply \Cref{lem:one-step} whenever some unassigned vertex $v$ satisfies $\ell_v\ge1/(4r)$. Each step assigns one vertex, preserves legality, and decreases $\Psi$, so $\Psi<M$ throughout. The process stops after at most $v(F)$ steps. At termination, every unassigned vertex $v$ satisfies $\ell_v<1/(4r)$.
\end{proof}

Fix a partial map $\varphi$ as in \Cref{clm:small-load-map}, and let $U$ be its set of unassigned vertices. If $U=\emptyset$, legality gives a copy of $F$. Otherwise, in the second stage we map the remaining vertices independently and uniformly to $V(H)$. We require only the edges of $F$ to be preserved; the final images need not span a clique.

\begin{claim}\label{lem:local-completion}
Suppose that the current partial map is legal and $\ell_v<1/(4r)$ for every $v\in U$. If the remaining vertices are mapped independently and uniformly to $V(H)$, then the probability that every edge of $F$ is successful is at least $(1-1/(2r))^m$.
\end{claim}

\begin{proof}
For every unfinished edge $e$, let $B_e$ be the event that $e$ fails. It is determined by the nonempty set $e\cap U$ of at most $r$ independent random variables, and $\Prb[B_e]=p_e$. We construct a dependency graph whose vertices are the events $B_e$. For distinct unfinished edges $e,f$, write $f\sim e$ if $e\cap f\cap U\ne\emptyset$, and join the corresponding events $B_e$ and $B_f$. For each $e$, the events not joined to $B_e$ depend only on variables outside $e\cap U$. Thus $B_e$ is independent of the collection of its nonneighbors, as required for a dependency graph.

By the asymmetric Lov\'asz Local Lemma of Erd\H{o}s and Lov\'asz~\cite{ErdosLovasz75}, in the form stated by Lu and Sz\'ekely~\cite[Lemma~2]{LuSzekely07}, the probability of avoiding all bad events is at least $\prod_e(1-x_e)$ if numbers $x_e\in[0,1)$ satisfy
\[
p_e\le x_e\prod_{f\sim e}(1-x_f)
\quad\text{for every unfinished edge }e.
\]

Take $x_e\coloneqq2p_e$. Since $e\cap U$ is nonempty, the bound $\ell_v<1/(4r)$ gives $0\le x_e=2p_e\le2\ell_v<1/(2r)<1$ for every $v\in e\cap U$. Every neighbor of $B_e$ shares an unassigned vertex with $e$. Summing the bounds on $\ell_v$ over the at most $r$ vertices of $e\cap U$ gives
\[
\sum_{f\sim e}p_f
\le\sum_{v\in e\cap U}\ell_v
<\frac{|e\cap U|}{4r}\le\frac14.
\]
The inequality $\prod_i(1-y_i)\ge1-\sum_i y_i$ for $y_i\in[0,1]$ therefore implies
\[
\prod_{f\sim e}(1-x_f)
\ge1-\sum_{f\sim e}x_f
\ge\frac12.
\]
Since $x_e=2p_e$, we have $x_e\prod_{f\sim e}(1-x_f)\ge2p_e\cdot\frac12=p_e$, which verifies the Local Lemma hypothesis. There are at most $m$ unfinished edges, so the Local Lemma yields
\[
\Prb\left[\bigcap_e\overline{B_e}\right]
\ge\prod_e(1-x_e)
\ge\left(1-\frac1{2r}\right)^m.
\]
Every completely assigned edge was already successful by legality, which proves the claim.
\end{proof}

To obtain a copy of $F$ from \Cref{lem:local-completion}, we must also exclude repeated images at vertices that do not lie in a common edge. Since $\varphi$ is legal, it is injective. For any pair of distinct vertices of $F$ with at least one vertex in $U$, their images agree with probability $1/n$. By the union bound, the probability of any repeated image is at most $\binom{v(F)}2/n$. An assignment embeds $F$ precisely when it is injective and every edge is successful. By \Cref{lem:local-completion} and the second inequality in \eqref{eq:n-choice}, we have
\begin{multline*}
\Prb[\text{the final map embeds }F]\\
\ge\Prb[\text{every edge is successful}]
-\Prb[\text{the final map is not injective}]
\ge\left(1-\frac1{2r}\right)^m
-\frac{\binom{v(F)}2}{n}>0.
\end{multline*}
Thus $H$ contains $F$, proving the embedding assertion.

For every fixed $0<q<c_rm^{-(r-1)/r}$, the embedding assertion gives $1-r!\ex(n,F)/n^r>q$ for all sufficiently large $n$. Since $r!\ex(n,F)/n^r\to\pi(F)$, letting $n\to\infty$ and then taking the supremum over $0<q<c_rm^{-(r-1)/r}$ gives $\pi(F)\le1-c_rm^{-(r-1)/r}$.
\end{proof}

\section*{Acknowledgements}
X.H. was supported by the National Key R\&D Program of China (2023YFA1010203), the
National Natural Science Foundation of China (12471336).
the Quantum Science and Technology-National Science and Technology Major Project (2021ZD0302902).
X.L. was supported by the Excellent Young Talents Program (Overseas) of the National Natural Science Foundation of China.
This research was carried out while the first author was visiting the School of Mathematical Sciences at the University of Science and Technology of China.

\section*{Declaration on the use of AI}
The authors used generative AI tools to assist in discussing proof strategies, checking proofs, and improving exposition.

\bibliographystyle{abbrv}
\bibliography{improved-sidorenko}

\end{document}